\documentclass{amsart}
\usepackage{amsfonts,amssymb,amsmath,amsthm}
\usepackage{url}
\usepackage{enumerate}

\newtheorem{thrm}{Theorem}[section]
\newtheorem{lem}[thrm]{Lemma}

\newtheorem{cor}[thrm]{Corollary}
\newtheorem{definition}[thrm]{Definition}

\newtheorem{ex}[thrm]{Example}
\numberwithin{equation}{section}

\author{C.A. Morales}

\address{Instituto de Matem\'atica\\
Universidade Federal do Rio de Janeiro\\
P. O. Box 68530 21945-970\\
Rio de Janeiro\\
Brazil.}
\email{morales@impa.br}

\begin{document}

\title[Equicontinuity on semi-locally connected spaces]{Equicontinuity on semi-locally connected spaces}

\begin{abstract}
We show that a homeomorphism of a semi-locally connected compact metric space is
equicontinuous if and only if
the distance between the iterates
of a given point and a given subcontinuum (not containing that point)
is bounded away from zero.
This is false for general compact metric spaces.
Moreover, homeomorphisms for which the conclusion of this result holds
satisfy that the set of automorphic points contains those points where the space is not
semi-locally connected.
\end{abstract}

\maketitle

\section{Introduction}
\noindent
A homeomorphism $f:X\to X$ of a compact metric space $X$ is
{\em distal} if
$$
\inf_{n\in\mathbb{Z}}d(f^n(x),f^n(y))>0
$$
for all distinct points $x,y\in X$.
The distal homeomorphisms were introduced by Hilbert as a generalization of rigid group of
motions \cite{a}, \cite{e}.

It is natural to generalize this definition by
replacing $y$ above by a compact subset $K$ (not containing $x$ of course);
and $d(f^n(x),f^n(y))$ by $dist(f^n(x),f^n(K))$ where
$$
dist(p,B)=\inf\{d(p,b):b)\in B\}, \quad\quad\forall p\in X, B\subset X.
$$
However, the resulting generalization turns out to be equivalent to equicontinuity (by Lemma \ref{lem1} below).
Recall that $f$ is {\em equicontinuous} if for every $\epsilon>0$ there is $\delta>0$ such that
$d(f^n(x),f^n(y))\leq\epsilon$ whenever $x,y\in X$ satisfy $d(x,y)\leq\delta$.

But the question arises what if we replace $y$ by another type of compact subsets like, for instance, a subcontinuum $C$ not containing $x$?
Recall that a {\em subcontinuum} is a nonempty compact connected subset of $X$.
The continuum theory \cite{n} has been considered elsewhere in the study of dynamical systems
\cite{acp}, \cite{k}.
This question suggests the study homeomorphisms $f: X\to X$ with the following property:

\begin{enumerate}
\item[(CW)] $\inf_{n\in \mathbb{Z}}dist(f^n(x),f^n(C))>0$ for every $x\in X$ and every
subcontinuum $C$ of $X$ with $x\notin C$.
\end{enumerate}

In our first result we will give topological conditions characterizing equicontinuity in terms of this property.
For this we shall use
the following definition by Whyburn \cite{w} (see also Jones \cite{j}).

\begin{definition}
We say that a metric space $X$ is {\em semi-locally connected at $p\in X$}
if for every open neighborhood $U$ of $p$ there is an open neighborhood
$V\subset U$ of $p$ such that $X\setminus U$ is contained
in the union of finitely many connected components of $X\setminus V$.
We denote by $X^{slc}$ the set of points at which $X$ is semi-locally connected.
We say that $X$ is {\em semi-locally connected} if $X^{slc}=X$.
\end{definition}

The class of semi-locally connected spaces is broad enough to include properly the locally connected ones
(like manifolds or more structured spaces).
In these spaces we shall obtain the following equivalence.

\begin{thrm}
\label{thB}
A homeomorphism of a semi-locally connected compact metric space is equicontinuous if and only if
it satisfies (CW).
\end{thrm}

In particular,
since there are distal homeomorphisms on the two-disk which are not
equicontinuous (e.g. Chapter 5 in \cite{a}),
and the two-disk is locally connected,
we obtain the following corollary.

\begin{cor}
\label{c1}
There are distal homeomorphisms which do not satisfy (CW).
\end{cor}

Notice that (CW) is also equivalent to equicontinuity on certain nonsemi-locally spaces 
like the totally disconnected ones (see Corollary 1.9 in \cite{agw}).
But in general (CW) is not equivalent to equicontinuity the following counterexample.

\begin{ex}
\label{ex1}
Define $X=\{re^{i\theta}\in \mathbb{R}^2:0\leq \theta\leq 2\pi,r\in C\}$
where $C$ is the ternary Cantor set in $[1,2]$.
Then, the map $f:X\to X$ defined by
$f(re^{i\theta})=re^{i(\theta+2\pi r)}$ is a homeomorphism of $X$.
This homeomorphism is an irrational rotation on the circle with radius $r$ for irrational $r$,
while it is a periodic rotation for rational $r$.
Using this we can see that $f$ is not equicontinuous.
Since every subcontinuum of $X$ is contained in one of the
circles $\{re^{i\theta}:0\leq\theta\leq 2\pi\}$ ($r\in C$) in which the action is an isometry,
we obtain that $f$ satisfies (CW).
\end{ex}

This counterexample motivates the search of similarities between (CW) and
the equicontinuous property.
For instance, since every equicontinuous homeomorphism is uniformly rigid \cite{gm},
it is natural to ask if every homeomorphism satisfying (CW) is
uniformly rigid (or at least rigid).
Another question comes from the following definition by
Veech \cite{v} (see also \cite{agn}):

\begin{definition}
If $f: X\to X$ is a homeomorphism of a compact metric space $X$,
a point $x\in X$ is called {\em almost automorphic}
if for every sequence $n_i\in \mathbb{Z}$ with $f^{n_i}(x)\to y$ for some $y\in X$
it holds that $f^{-n_i}(y)\to x$. We denote by $A(f)$ the set of almost automorphic points of $f$.
\end{definition}

Theorem 3 in \cite{agn} implies that every equicontinuous homeomorphism $f: X\to X$ of a compact metric space $X$ satisfies $A(f)=X$.
Is this property true (or at least $A(f)\neq\emptyset$) for every homeomorphisms $f$ satisying (CW)?
A positive answer is supported by Example \ref{ex1}
in which $A(f)=X$ (as explained in Example 9 in \cite{agn}).
Another positive answer is supported by our second result.

\begin{thrm}
\label{thD}
If $f: X\to X$ is a homeomorphism satisfying (CW) of a compact metric space $X$,
then $X^{slc}\subset A(f)$.
\end{thrm}

Consequently,

\begin{cor}
Let $X$ be a compact metric space.
If $X^{slc}\neq\emptyset$, then $A(f)\neq\emptyset$ for
every homeomorphism $f: X\to X$ satisfying (CW).
\end{cor}

On the other hand, $X^{lsc}= A(f)$ is false in general for homeomorphisms $f$ satisfying (CW).
Indeed, every equicontinuous homeomorphism satisfies (CW) with $A(f)=X$ thus
it suffices to take the identity of the space described in Example \ref{ex1}.

Theorem \ref{thD} has some other interesting consequences.
For instance,
recall that a homeomorphism $f: X\to X$ is {\em transitive} if the there is $x\in X$ such that
the orbit $\{f^n(x):n\in\mathbb{Z}\}$ of $x$ is dense in $X$.

\begin{cor}
\label{c0}
Let $X$ be a compact metric space. If $X^{slc}\neq\emptyset$,
then a transitive homeomorphism of $X$ satisfies (CW) if and only if it is equicontinuous.
\end{cor}

\begin{proof}
Apply Theorem \ref{thD}, Corollary 14 in \cite{agn} and
the fact that every homeomorphism satisfying (CW) is distal.
\end{proof}

A good question is if this corollary is false without the assumption $X^{slc}\neq\emptyset$.

By Theorem \ref{thD} we can also conclude that a compact metric space exhibiting
homeomorphisms satisfying (CW) but without almost automorphic points
cannot be semi-locally connected at any point.
However, as already mentioned, we don't know if there are such kind of
pathological homeomorphism.

The author would like to thank Professor D. Obata for helpful conversations related to Lemma \ref{lem1}.

\section{Proof of theorems \ref{thB} and \ref{thD}}

\noindent
We start with the following lemma.

\begin{lem}
\label{l1}
If $f:X\to X$ is an equicontinuous homeomorphism
of a compact metric space $X$,
then $\inf_{n\in\mathbb{Z}} dist(f^n(x),f^n(K))>0$
for every $x\in X$ and every nonempty compact subset 
$K$ with $x\notin K$.
\end{lem}

\begin{proof}
Suppose by contradiction that $f$
does not satisfy the required property.
Then, there exist $x\in X$ and a nonempty compact subset 
$K$ with $x\notin K$ such that
$
\inf_{n\in\mathbb{Z}} dist(f^n(x),f^n(K))=0.
$
From this we obtain sequences $b_k\in K$
and $n_k\in\mathbb{Z}$ such that
$$
\lim_{k\to\infty}d(f^{n_k}(x),f^{n_k}(b_k))=0.
$$
Since $K$ is compact, we can assume
that $b_k\to b$ for some $b\in K$.
Since $x\notin K$, we have $\epsilon=\frac{d(x,b)}{3}>0$.
For this $\epsilon$ we
let $\delta$ be given by the equicontinuity of $f$.

Since $b_k\to b$, there is $k\in\mathbb{N}$ such that
$d(b_k,b)\leq\epsilon$ and $d(x_k,y_k)\leq\delta$
where $x_k=f^{n_k}(x)$ and $y_k=f^{n_k}(b_k)$.
It follows that $d(f^{-n_k}(x_k),f^{-n_k}(y_k))\leq\epsilon$ by the equicontinuity property.
Since $f^{-n_k}(x_k)=x$ and $f^{-n_k}(y_k)=b_k$, we obtain $d(x,b_k)\leq \epsilon$ and so
$
d(x,b)\leq d(x,b_k)+d(b_k,b)
\leq 2\epsilon= \frac{2}{3}d(x,b)
$
which is absurd. This ends the proof.
\end{proof}

We also need the following results.

\begin{lem}
\label{l2}
Let $f: X\to X$ be a homeomorphism of a compact metric space $X$. If $p\in X$ satisfies
$\inf_{n\in\mathbb{Z}} dist(f^n(p),f^n(K))>0$
for every nonempty compact subset 
$K$ with $p\notin K$, then for every $\epsilon>0$ there is
$\delta>0$ such that
$$
y\in X, n\in\mathbb{Z} \mbox{ and }d(f^n(p),f^n(y))\leq\delta\quad \Longrightarrow \quad d(p,y)\leq\epsilon.
$$
\end{lem}

\begin{proof}
Fix $\epsilon>0$.
Denote by $B(\cdot,\cdot)$ and $B[\cdot,\cdot]$ the open and closed ball operations in $X$ respectively.
We can assume $X\setminus B(p,\epsilon)\neq\emptyset$
for, otherwise, the conclusion is trivial.
Since
$K=X\setminus B(p,\epsilon)$ is compact not containing $p$,
there is $\delta$ satisfying
$$
0<\delta<\inf_{n\in\mathbb{Z}}dist(f^n(p),f^n(X\setminus B(p,\epsilon))).
$$
It follows that
$B[f^n(p),\delta]\subset f^n(B(p,\epsilon))$ for every $n\in\mathbb{Z}$.
Then, if $y\in X$ and $n\in\mathbb{Z}$ satisfy
$d(f^n(p)f^n(y))\leq\delta$,
$f^n(y)\in B[f^n(p),\delta]$ so $f^n(y)\in f^n(B(p,\epsilon)))$
thus $y\in B(p,\epsilon)$ yielding $d(p,y)\leq\epsilon$.
\end{proof}

Now some standard terminology.
A subset $F\subset \mathbb{Z}$ is called {\em syndetic}
if there is $l>0$ such that
$[n,n+l]\cap F\neq\emptyset$ for every $n\in \mathbb{Z}$.
Clearly $F$ is syndetic if and only if $-F=\{-n:n\in F\}$ is.
If $f:X\to X$ is a homeomorphism,
we say that $x\in X$ is {\em almost periodic}
if for every neighborhood $U$ of $x$ there is a syndetic
$F\subset \mathbb{Z}$ such that $f^n(x)\in U$ for every $n\in F$.
We say that {\em $f$ is pointwise almost periodic} if every point is almost periodic.

On the other hand, $x\in X$ is {\em locally almost periodic}
if for every neighborhood $U$ of $x$ there are a neighborhood
$V\subset U$ of $x$ and a syndetic $F\subset \mathbb{Z}$
such that $f^n(V)\subset U$ for every $n\in F$.
We say that {\em $f$ is locally almost periodic} if
every point $x$ is locally almost periodic.

The next lemma gives a sufficient condition for an almost periodic point to be locally almost periodic.

\begin{lem}
\label{l4}
Let $p$ be an almost periodic point of a homeomorphism $f: X\to X$
of a compact metric space $X$.
If for every $\epsilon>0$ there is
$\delta>0$ such that
$$
y\in X, n\in\mathbb{Z} \mbox{ and }d(f^n(p),f^n(y))\leq\delta\quad \Longrightarrow \quad d(p,y)\leq\epsilon,
$$
then $p$ is locally almost periodic.
\end{lem}

\begin{proof}
Fix a neighborhood $U$ of $p$.
Take $\epsilon>0$ such that $B[p,\epsilon]\subset U$.
For this $\epsilon$ choose $\delta$ from the hypothesis of the lemma.
We can assume that $\delta<\epsilon$.
Since $p$ is almost periodic,
there is a syndetic $F\subset \mathbb{Z}$
such that $f^n(p)\in B[p,\frac{\delta}{2}]$ for all
$n\in F$.
Define $W=B[p,\frac {\delta}{2}]$.
Then, $W$ is a neighborhood of $p$ contained in $U$
(for $\delta<\epsilon$).
If $y\in W$ and $n\in F$, then
$d(y,p)\leq\frac{\delta}{2}$ and $d(p,f^n(p))\leq\frac{\delta}{2}$. Then,
$d(f^n(p),y)\leq\delta$ so
$d(f^n(p),f^n(f^{-n}(y)))=d(f^n(p),y)\leq\delta$
thus $d(p,f^{-n}(y))\leq\epsilon$ for all $n\in F$.
We conclude that $f^n(W)\subset U$ for every $n\in -F$.
Since $-F$ is syndetic,
we are done.
\end{proof}

Next we present two characterizations of equicontinuous
homeomorphisms the first of which is well known (Theorem 1 in \cite{g}).

\begin{lem}
\label{l3}
A homeomorphism of a compact metric space
is equicontinuous if and only if it is distal and locally almost periodic.
\end{lem}

The second characterization (stated below) is
related to Theorem 3.4 in \cite{loyz} where it is proved
that a homeomorphism is equicontinuous if and only if
the induced map in the hyperspace is distal.

\begin{lem}
\label{lem1}
A homeomorphism $f: X\to X$ of a compact metric space $X$ is equicontinuous if and only if 
$\inf_{n\in\mathbb{Z}} dist(f^n(x),f^n(K))>0$ for all $x\in X$ and all nonempty compact subset $K$ of $X$
with $x\notin K$.
\end{lem}

\begin{proof}
The necessity follows
from Lemma \ref{l1}.
For the sufficiency,
assume that $\inf_{n\in\mathbb{Z}} dist(f^n(x),f^n(K))>0$ for all $x\in X$ and all compact nonempty subset $K$
with $x\notin K$.
Then, $f$ is distal
and so pointwise almost periodic (e.g. Corollary 4 p. 68 in \cite{a}).
Moreover, by Lemma \ref{l2}, for every
$p\in X$ and every $\epsilon>0$ there is
$\delta>0$ such that
$$
y\in X, n\in\mathbb{Z} \mbox{ and }d(f^n(p),f^n(y))\leq\delta\quad \Longrightarrow \quad d(p,y)\leq\epsilon.
$$
Then, Lemma \ref{l4} implies that
$f$ is locally almost periodic.
Since $f$ is distal, $f$ is equicontinuous by Lemma \ref{l3}.
This completes the proof.
\end{proof}

The last ingredient stated below gives the link between (CW) and $X^{slc}$.

\begin{lem}
\label{cw}
Let $f:X\to X$ be a homeomorphism satisfying (CW)
of a compact metric space $X$.
Then, $\inf_{n\in\mathbb{Z}}dist(f^n(p), f^n(K))>0$
for every $p\in X^{slc}$ and every compact subset $K$ of $X$ with $p\notin K$.
\end{lem}

\begin{proof}
Suppose by contradiction that
$\inf_{n\in\mathbb{Z}}dist(f^n(p),f^n(K))=0$ for some $p\in X^{slc}$ and some compact subset $K$ of $X$ with $p\notin K$.
Then, there are sequences $n_k\in \mathbb{Z}$ and $y_k\in K$ such that
\begin{equation}
\label{eq1}
\lim_{k\to\infty}d(f^{n_k}(p),f^{n_k}(y_k))=0.
\end{equation}

Now fix $z\in K$.
Since $p\notin K$, there is an open neighborhood $U_z$ of $p$
such that $z\in X\setminus Cl(U_z)$ where $Cl(\cdot)$ denotes the closure operation.
Then, $\{X\setminus Cl(U_z):z\in K\}$ is an open covering of $K$ which is compact so
there are finitely many
points $z_1,\cdots, z_r\in K$ such that
$$
K\subset \bigcup_{i=1}^r(X\setminus Cl(U_{z_i})).
$$
Since $y_k\in K$ for $k\in \mathbb{N}$, we can assume by passing to a subsequence
if necessary that there is a fixed index $1\leq i\leq r$ such that
$y_k\in X\setminus Cl(U_{z_i})$ for all
$k\in \mathbb{N}$.

But $p\in X^{slc}$. Then,
there exist
an open neighborhood $V$ of $p$ and finitely many
connected components $C_1,\cdots, C_l$ of $X\setminus V$
such that
$$
X\setminus U_{z_i}\subset \bigcup_{j=1}^lC_j.
$$
Since $y_k\in X\setminus Cl(U_{z_i})\subset X\setminus U_{z_i}$ for all $k\in \mathbb{N}$,
we can assume again up to passing to a subsequence if necessary that there is another
fixed index $1\leq j\leq l$ such that $y_k\in C_j$
for all $k\in\mathbb{N}$.

On the one hand, $C_j$ is a subcontinuum of $X$ with $p\not\in C_j$ hence (CW) implies
$$
\inf_{n\in \mathbb{Z}}dist(f^n(p),f^n(C_{j}))>0.
$$
But, on the other hand,
$$
\inf_{n\in \mathbb{Z}}dist(f^n(p),f^n(C_{j})\leq \inf_{k\in\mathbb{N}}d(f^{n_k}(p),f^{n_k}(y_k))
$$
so (\ref{eq1}) implies
$$
\inf_{n\in \mathbb{Z}}dist(f^n(p),f^n(C_{j})=0
$$
which is absurd. This completes the proof.
\end{proof}

Now we can prove our results.

\begin{proof}[Proof of Theorem \ref{thB}]
Let $f: X\to X$ be a homeomorphism of a semi-locally connected compact metric space $X$, i.e.,
$X=X^{slc}$.
If $f$ is equicontinuous, then $f$ satisfies (CW) by Lemma \ref{l1}.
Conversely, if $f$ satisfies (CW), then
Lemma \ref{cw} and $X=X^{lc}$ imply
$\inf_{n\in\mathbb{Z}}dist(f^n(p), f^n(K))>0$ for every $x\in X$ and every compact subset $K$ of $X$ with $x\notin K$.
Then, $f$ is equicontinuous by Lemma \ref{lem1}. 
\end{proof}

\begin{proof}[Proof of Theorem \ref{thD}]
Let $f: X\to X$ be a homeomorphism satisfying (CW) of a compact metric space $X$.
If $p\in X^{slc}$, then
$\inf_{n\in\mathbb{Z}}dist(f^n(p), f^n(K))>0$ for every compact subset $K$ of $X$ with $p\notin K$
by Lemma \ref{cw}.
From this and Lemma \ref{l2} we obtain that
for every $\epsilon>0$ there is
$\delta>0$ such that
$$
y\in X, n\in\mathbb{Z} \mbox{ and }d(f^n(p),f^n(y))\leq\delta\quad \Longrightarrow \quad d(p,y)\leq\epsilon.
$$
To finish we prove that this property implies that $p$ is almost authomorphic.
Indeed, take $y\in X$ and a sequence $n_i\in \mathbb{Z}$ such that $f^{n_i}(p)\to y$.
Fix $\epsilon$ and let $\delta$ be given by the above property.
Since $f^{n_i}(p)\to y$, there is $i_0>0$ such that
$d(f^{n_i}(p),y)\leq \delta$ for all $i\geq i_0$.
It follows that $d(f^{n_i}(p),f^{n_i}(f^{-n_i}(y)))=d(f^{n_i}(p),y)\leq \delta$
and so $d(p,f^{-n_i}(y))\leq\epsilon$ for all $i\geq i_0$.
Hence $f^{-n_i}(y)\to p$ and the proof follows.
\end{proof}

\end{document}